\documentclass[11pt,reqno]{amsart}

\usepackage[T1]{fontenc}
\usepackage{verbatim}
\usepackage{graphicx}

\usepackage{color}
\definecolor{MyLinkColor}{rgb}{0,0,0.4}

\newcommand{\tr}{\mathop{\rm tr}\nolimits}

\newcommand{\h}{\rho}
\newcommand{\0}{\Omega}

\newcommand{\p}{\partial}
\newcommand{\w}{\widetilde}
\newcommand{\ov}{\overline}

\newcommand{\G}{\Gamma}

\newcommand{\A}{\mathcal{A}}
\newcommand{\B}{\mathcal{B}}
\newcommand{\oo}{\mathcal{O}}
\newcommand{\ww}{\mathcal{W}}
\newcommand{\zz}{\mathcal{Z}}

\newcommand{\cF}{\mathcal{F}}

\newcommand{\kL}{\mathcal{L}}

\newcommand{\Q}{\mathcal{Q}}
\newcommand{\cR}{\mathcal{R}}
\newcommand{\T}{\mathcal{T}}
\newcommand{\U}{\mathcal{U}}
\newcommand{\V}{\mathcal{V}}
\newcommand{\X}{\mathbb{X}}
\newcommand{\Y}{\mathbb{Y}}

\newcommand{\R}{\mathbb{R}}
\newcommand{\s}{\mathbb S}
\newcommand{\N}{\mathbb{N}}
\newcommand{\Z}{\mathbb{Z}}

\newtheorem{thm}{Theorem}[section]
\newtheorem{prop}[thm]{Proposition}
\newtheorem{lemma}[thm]{Lemma}

\theoremstyle{remark} 
\newtheorem{rem}[thm]{Remark}

\numberwithin{equation}{section}
\title[On the well-posedness of a water-mud system]
{On the well-posedness of a mathematical model describing  water-mud interaction}

\subjclass[2010]{35A39, 76A05, 35R57, 76T99.}
\keywords{Classical solution; non-Newtonian fluid; two-phase moving boundary problem.}

\author[Joachim Escher]{Joachim Escher}
\author[Anca--Voichita Matioc]{Anca--Voichita Matioc}
\address{Institut f{\"u}r Angewandte Mathematik,  Leibniz Universit{\"a}t  Hannover, Welfengarten 1, 30167 Hannover, Deutschland. }
\email{escher@ifam.uni-hannover.de}
\email{matioca@ifam.uni-hannover.de}
\usepackage[colorlinks=true,linkcolor=MyLinkColor,citecolor=MyLinkColor]{hyperref} 

\begin{document}
\begin{abstract}
In this paper we consider  a mathematical model describing the two-phase interaction
between water and mud in a   water canal when the width of the canal is small compared to its depth. 
The mud is treated as a non-Netwonian fluid  and the interface between the mud and fluid is allowed to move under the influence of gravity and surface tension.

We reduce the mathematical formulation, for small boundary and initial data, to a fully nonlocal and nonlinear  problem and prove   its local well-posedness by using  abstract parabolic theory. 
\end{abstract}

\maketitle

\section{Introduction}
We consider herein a mathematical model describing the evolution of a two-phase   system consisting of mud and water.
This problem is of immense physical importance as it is related to sediment transportation in coastal and estuarine regions.
Therefore, a large number of  experiments and theoretical investigations have been and are still dedicated to better understand two-phase  mud-water systems. 

We shall treat the mud as being a non-Newtonian fluid, which reflects the rheological behaviour of fluid muds, cf. \cite{T}, as the viscosity changes with the rate of strain.
The viscosity function for the mud, though general,  is assumed to satisfying certain, not to restrictive, conditions.
These conditions appear in the literature in the context of the  non-Newtonian Navier-Stokes problem \cite{BP} as well for the Stokesian Hele-Shaw flow \cite{EM1} and it was proven that these are
sufficient conditions to ensure the local well-posedness of these problems.
A diffuse interface model for the flow of two viscous incompressible
Newtonian fluids in a bounded domain have been analyzed at the level of weak solutions in \cite{A}, when considering  fluids with different densities.
Assuming a sharp interface between the fluids, which are modeled by the full Navier-Stokes equations and have constant densities,  the new formulation has been studied  in \cite{PS, PS1} by using maximal regularity for the linearized problems.

Numerically, this physical setting has also received plenty of attention leading to a large number of models and experiments investigating the entrainment and settling of the fluid mud, cf. \cite{Hsu, KW, W}.
Other  experimental studies are concerned with the damping effect of fluid muds on the waves traveling at the water surface, see e.g. \cite{DL, WG}.

Fluid mud is a highly concentrated suspension of fine grained sediments and is treated herein as being a homogeneous fluid with general shear dependent viscosity.
Furthermore, the width of the canal is assumed small compared to its length and hight, so that, width averaging the Navier-Stokes equations we can 
approximate the flow in the water and mud region by  two-dimensional  linear and nonlinear Darcy's laws, respectively, cf. \cite{KSP}.
In this way we incorporate the viscous behaviour of the fluid mud into our model but loose the generality one has when working with the full Newtonian and non-Newtonian  Navier-Stokes equations.
Assuming also that there is a sharp interface between mud and water, we are lead to a coupled problem 
describing the evolution of the water-mud system, and, 
if the mud is replaced by a Newtonian fluid, we rediscover the Muskat problem \cite{ CCG, EM, Y} which is a current research topic.      
We note also that there are recent investigations \cite{MS} of the least squares formulation for a two-phase coupled problem with Stokes flow in the one subdomain and a linear Darcy flow model for the second fluid. 

Having a general viscosity function makes the analysis more involved compare to the Newtonian case \cite{ CCG, EM, Y}.
Using elliptic theory for quasilinear equations though, we are able to formulate the problem with 
the help of a nonlinear and nonlocal operator depending on the function  parameterizing the sharp interface between the mud and water and its time derivative.   
Taking advantage of the properties of the flow, that is the volume of fluids are preserved,  the implicit function theorem
allows us to reformulate the problem as an abstract evolution equation for small boundary and initial data.
Under a smallness condition for the time derivative of this function, which is necessary only if the mud is a shear thinning fluid, 
we prove then by using parabolic theory \cite{L} existence and uniqueness of local solutions for our problem.

The outline of the paper is as follows. In Section \ref {S2} we present the mathematical model and state the main result of this paper, Theorem \ref{thmain}.
In Section \ref{S3} we  transform the coupled system of equations obtain in Section \ref{S2} and express the problem as a single operator equation.
The proof of the main result is found in Section \ref{S4}.

\section{The mathematical model and the main result}\label{S2}
We present now the mathematical model describing the two-phase interaction between water and mud in a  water canal with small width.
The subscript $m$ is used when referring to the mud and which occupies the domain $\Omega_m(t)$, respectively $w$ for the water located in $\Omega_w(t).$
These two fluids are separated by the interface $\Gamma(t)$ which evolves in time and is to be determined as a part of the problem.
The physical relevant problem is essentially three-dimensional, but it may be approximated by a two-dimensional mathematical model.
 We refer to \cite{KSP} for a deduction of a generalized Darcy's law for non-Newtonian fluids, which serves as a approximation of the non-Newtonian  Navier-Stokes equations.
Thus, we restrict our considerations to the situation when $\Omega_{m}(t), \Omega_{w}(t)\subset \R^2.$

The mud-water problem is a potential flow, in the sense that the velocity fields $\vec v_m, \vec v_w$ satisfy Darcy laws. 
The motion of the mud  is governed by the nonlinear  Darcy's law
\begin{equation}\label{darc+}
 \vec v_m=-\frac{\nabla u_m}{ \mu_m (|\nabla u_m|^2)} \qquad \text{in} \quad \Omega_m(t)
\end{equation}
 and in the water domain we set  
\begin{equation}\label{darc-}
 \vec v_w=-\frac{1}{\mu_w}\nabla u_w  \phantom{aaa}\qquad \text{in}\quad \Omega_w(t),
\end{equation}
$\mu_w$ being the constant viscosity of the water.
The effective viscosity $\mu_m$ is defined (cf. \cite{KSP}) by
\begin{equation}\label{eq:visc}
 \mu_m(r):=\left(\int\limits_{-1}^1 \frac{s^2}{\w \mu(rs^2)}\, ds\right)^{-1} \qquad\text{for  $r\geq 0$}.
\end{equation} 
With  $\mu \in C^\infty([0,\infty), (0, \infty))$ denoting the shear-rate dependent viscosity of the mud, 
we let $\w\mu:= c \mu\circ[r\mapsto r\mu^2(r)]^{-1}$, 
whereby $c$ is a constant depending on the gap width  of the canal. 
The invertibility of the mapping $[[0,\infty)\ni r\mapsto r\mu^2(r)\in[0,\infty]$ is guaranteed if we restrict our considerations to fluids 
having viscosity functions which satisfy:
\begin{equation}\label{viscosity}
m\leq \mu \leq M\quad\text{and}\quad m\leq \mu(r)+2r\mu'(r)\leq M
\end{equation}
for some positive constants $m, M.$
It follows  from \eqref{eq:visc} that $\mu_m$ inherits similar properties as $\mu,$ that is $\mu_m \in C^\infty([0,\infty), (0, \infty))$
and $\mu_m$ satisfies 
\begin{equation}\label{viscosity'}
m\leq \mu_m \leq M\quad\text{and}\quad m\leq \mu_m(r)-2r\mu_m'(r)\leq M
\end{equation}
possibly with different constants. 
We refer to \cite{EM1} for precise calculations.
Moreover, the scalar functions  $u_w$ and $u_m$ are defined by
\begin{equation}\label{potm}
\begin{array}{rcccc}
 u_m:= p_m+g\rho_m y  \quad \text{in $\Omega_m(t)$}\qquad \text{and} \qquad u_w:= p_w+g\rho_w y  \quad \text{in  $\Omega_w(t)$},
\end{array}
\end{equation}
whereby  $p_m, p_w$ are the dynamic pressures,   $\rho_m, \rho_w$ denote the constant densities of the fluids,
 $g$ is the gravity constant, and  $y$ is the height coordinate.
Relations \eqref{viscosity} are  satisfied, besides by all Newtonian fluids, by the Bingham model with  dynamic viscosity
\begin{equation}\label{hectorite}
\mu(r)=\mu_\infty+\frac{\tau_0\beta}{1+\beta r}
\end{equation}
if and only if  $\beta\tau_0<4\mu_\infty.$
For this reason, suspensions of hectorite can be considered in our model (cf. \cite{T}). 

Additionally to \eqref{potm},  we assume that the fluids are incompressible, that means
\begin{equation}\label{incom}
\begin{array}{rccc}
 \nabla \cdot \vec v_m&=&0 \quad \text{in $ \Omega_m$} \qquad\text{and}\qquad
\nabla \cdot \vec v_w=0 \quad \text{in $ \Omega_w$}.
\end{array}
\end{equation}
The surface tension plays a major role in interfacial phenomena as a restoring force.
 It  compensates the pressure difference across the interface $\Gamma(t)$ according to  Laplace-Young's condition, 
 so that we get the following equation  at the interface between the fluids:
\begin{equation}\label{LY}
 u_w-u_m=\gamma \kappa_{\Gamma(t)}+g(\rho_w- \rho_m)y \qquad \text{on} \quad \Gamma(t),
\end{equation}
with $\gamma$ being the surface tension coefficient at the interface, and $\kappa_{\Gamma(t)}$ is the curvature of $\Gamma(t).$
Equalizing the normal components of the two velocity fields on the interface $\Gamma(t)$,
 we also obtain two kinematic boundary conditions at this boundary:
\begin{equation}\label{kc}
 V(t, \cdot)= \langle \vec v_m (t, \cdot), \nu(t,\cdot)\rangle 
= \langle \vec v_w (t, \cdot), \nu(t,\cdot)\rangle \qquad \text{on} \quad \Gamma(t),
\end{equation}
which means that the interface moves along with the fluids (the interface consists  of the same particles at all times).
We denoted here by $\nu(t)$ the unit normal at $\Gamma(t)$ which points into $\Omega_w(t),$ and $V(t)$ is the normal velocity of the interface.
To complete the model, we  presuppose
that the bottom $\Gamma_ {-1}=[y=-1]$ of the canal  is impermeable, that is 
\begin{equation}\label{bot}
 u_{m,2}=0 \qquad \text{on} \quad \Gamma_{-1},
\end{equation}
where $u_{m,i}$  and $u_{m,ij},$ $1\leq i,j\leq2,$ stand the first and second order derivatives of $u_m$ (similarly for $u_w$).
Also, we prescribe Dirichlet boundary conditions at the interface $\Gamma_1:=[y=1]$, which is presupposed to be located above $\Gamma(t)$, that is: 
\begin{equation}\label{top}
u_w=h(t, x) \qquad \text{on} \quad \Gamma_{1}.
\end{equation}
Lastly, we assume  the interface at time $t=0$ to be known
\begin{equation}\label{int}
\Gamma(0)=  \Gamma_{0}.
\end{equation}

System \eqref{darc+}--\eqref{int} is a two-phase moving boundary problem. 
The main interest is to determine the motion of the interface $\Gamma(t)$ separating the fluids.
If we know $\Gamma(t),$ then we can find the potentials $u_m$ and $u_w$ by solving elliptic mixed boundary value problems.

\subsection{Parameterizing the boundary}
In the following we restrict our considerations to  horizontally $2\pi-$periodic flows.
The unknown interface $\Gamma(t),$ which is to be determined, is assumed to be parametrized by a function from the set
\[
 \U:=\{ f\in C(\s): ||f||_{C(\s)}<1/2\},
\]
where $\s\cong R/2\pi\Z$ is the unit circle. 
 Let $\alpha \in (0,1)$ be fixed for the remainder of this paper and define   $\X:=h^{4+\alpha}_0(\s)$ 
when we incorporate surface tension effects into our problem,  respectively $\X:=h^{2+\alpha}_0(\s)$ when $\gamma=0.$  
Moreover,  we let $\Y:=h^{1+\alpha}_0(\s).$
The small H\"older space $h^{m+\beta}(\s),$ $m \in \N$ and $ \beta \in (0,1)$, is defined as being the closure of $C^\infty(\s)$
in $C^{m+\beta}(\s).$ 
Particularly, small H\"older spaces are densely and, by Arzel\`a-Ascoli's theorem, also compactly embedded in those with lower exponent.
A further important feature in our context is that they  are stable under the continuous interpolation functor $(\cdot|\cdot)_{\theta,\infty}^0,$ cf. \cite{L}.
The subspaces $h^{m+\beta}_0(\s) $  consist only of the function in $h^{m+\beta}(\s) $ with integral mean zero.
Our choice for $\X$ and $\Y$ is due to the following observation:  if $\Gamma(t)$ is at rest the flat line $y=0,$ because there is no flux over $\Gamma_{-1},$ cf. \eqref{bot},
the volume of fluid mud of a solution $\Gamma(t)$ should match the one at rest. 
  Lastly, we define $\V:=\U\cap \X$ and observe that is an open neighborhood of the zero function in $\X.$ 

By our  modeling considerations,  the two fluids completely fill the  domain $\Omega:=\s\times (-1, 1).$
If $f:[0,T]\to \V$, with $T>0$, is a function describing the  evolution of the interface between the fluids, then at any 
time $t\in [0,T]$ we have that $\Omega_m(t)=\Omega_m(f(t)),$ $\Omega_w(t)=\Omega_w(f(t))$, and $\G(t)=\G(f(t)),$ 
where, given $h\in \V$, $\Gamma(h):=\{(x, h(x): \, x\in \s\}$ and 
\begin{equation*}
\begin{array}{lllcc}
 &\Omega_m(h):=\{ (x,y)\in \Omega: \, -1<y<h(x)\}, \qquad\Omega_w(h):=\{ (x,y)\in \Omega: \, h(x)<y<1\}.
\end{array}
\end{equation*}
In order to reformulate the kinematic conditions \eqref{kc} we follow the evolution of a particle  $(x(t), y(t))$ on the interface $\Gamma(f(t)).$ 
By the definition of $\Gamma(f(t))$, we know that  $y(t)= f(t,x(t))$, and differentiating this relation with respect to time yields $\p_tf=(-f',1)\vec v,$
where $\vec v(t)$ is the velocity of $\Gamma(t).$ 
This shows that the normal velocity $V(t)$ of the interface $\Gamma(t)$ is given by $V(t)=\p_t f(t)/\sqrt{1+f'^2(t)}.$

Therewith, the evolution  of the water-mud system 
\eqref{darc+}--\eqref{int}   is described by the following system of equations:
\begin{equation}\label{newpb}
\left\{
 \begin{array}{rllllllllllll}
\Delta u_w&=& 0 \qquad &\text{in} \quad   \Omega_w(f), \\
\nabla \left(\displaystyle\frac{\nabla u_m}{\mu_m(|\nabla u_m|^2)}\right)&=& 0 \qquad &\text{in} \quad  \Omega_m(f), \\
u_w &=& h \qquad &\text{on} \quad  \Gamma_{1}, \\
 u_{m,2} &=& 0  \qquad &\text{on} \quad  \Gamma_{-1}, \\
u_w-u_m&=& \gamma \kappa_{\Gamma(f)}+g(\rho_w-\rho_m)f \qquad &\text{on} \quad  \Gamma(f),\\
\p_tf= -\displaystyle\frac{1}{\mu_w}\sqrt{1+f'^2}\p_\nu u_w&=& -\displaystyle\frac{\sqrt{1+f'^2}}{\mu_m(|\nabla u_m|^2)}\p_\nu u_m 
\qquad &\text{on} \quad  \Gamma(f),\\
f(0) &=&f_0,
  \end{array}
\right.
\end{equation}
where  $f_0\in  \V$ is chosen such $\Gamma_0=\Gamma(f_0).$
We shall assume that the Dirichlet boundary data $h$ is continuous in time, more precisely
\begin{equation}\label{bcthm}
h\in C([0, \infty), h^{2+\alpha}(\s)).
 \end{equation}
 For the sake of simplicity, we identify  functions defined on $\G_{1}$ and $\G_{-1}$ with functions on the unit circle $\s.$

We shall call a triple $(f, u_w, u_m)$   classical H\" older solution of \eqref{newpb} on $[0,T]$ if 
\begin{align*}
 &f\in C([0,T], \V)\cap C^1([0,T], \Y),\\
 & u_w(t, \cdot) \in h^{2+\alpha}(\Omega_w(f(t))), \quad u_m(t, \cdot) \in h^{2+\alpha}(\Omega_m(f(t))), \, \, t\in [0,T],
\end{align*}
and if $(f, u_w, u_m)$ satisfies \eqref{newpb} pointwise. 
Given functions  $f_1, f_2 \in C(\s)$ with $f_1<f_2$, we let $\0_{f_1,f_2}:=\{(x,y)\,:f_1(x)<y<f_2(x)\}$. 
The space 
 $h^{2+\alpha}(\0_{f_1,f_2})$ 
is defined then as being the closure of the 
 smooth functions $C^{\infty}(\ov {\Omega_{f_1, f_2}}) $  
in the Banach space $C^{2+\alpha}(\ov {\0_{f_1,f_2}})$.
The main result of this paper is the following theorem:

\begin{thm}[Local well-posedness]\label{thmain}
 Let $\gamma \in [0, \infty), \, \ov h\in \R$ and assume that
\begin{align}\label{asump}
 \gamma >0 \quad \text{or} \quad \rho_m >\rho_w.
\end{align}
Then, there exist $\oo_f$ an open neighborhood of zero in $\X,$  $\zz$ an open neighborhood of zero in  $\Y$, and $\oo_{\ov h}$ an open neighborhood of $\ov h$ in $h^{2+\alpha}(\s)$ such that 
if $f_0\in \oo_f$, $h$ satisfies \eqref{bcthm}, and $h(0)\in \oo_{\ov h}$ then there exists $T>0$ and a unique maximal solution $f:[0, T)\to \oo_f$ satisfying  $\p_tf(t)\in \zz$ and $h(t)\in \oo_{\ov h}$ for all $t \in [0, T).$
Moreover, if $\mu_m'\geq 0$ then $\zz=\Y.$   
\end{thm}

\begin{rem}\label{R:1}
In contrast to the constant viscosity case studied in \cite{EM} where, at any time $t_0$ where the solution exists, the time derivative  $\p_tf(t_0)$ 
 is uniquely determined by the boundary conditions and  $f(t_0),$
the situation considered herein is more complex and it is not clear whether this feature is preserved when the viscosity $\mu_m$ is decreasing.
Particularly, when considering  the fluid mud model \eqref{hectorite} from \cite{T}, $\p_tf$ is uniquely determined by $f$ under the smallness condition  $\p_tf\in \zz, $ cf. Proposition \ref{P:4}.
\end{rem}

\section{The  transformed problem}\label{S3}
The difficulty when dealing with problem \eqref{newpb} is mainly due to the fact that the equations  of \eqref{newpb}
are defined on domains which evolve in time.
To overcome this problem, we shall transform the system by writing its equations on fixed domains.
This will transform though the differential operators into operators with coefficients depending on the unknown $f$. 
 Then, we will introduce  solution operators  to linear and quasilinear elliptic 
boundary value problems involving these transformed differential operators and study their regularity properties.
At the end we obtain at an operator equation for the function $f$ and its time derivative. 

To simplify our notation, we introduce first the quasilinear operator $\Q: C^2(\Omega_m) \to C(\Omega_m)$  by setting
\[
 \Q u_m:= \nabla \left(\displaystyle\frac{\nabla u_m}{\mu_m(|\nabla u_m|^2)}\right) \quad \text{for $u_m\in C^2(\Omega_m).$}
\]
For $u_m\in C^2(\Omega_m)$,  one easily  computes that $\Q u=a_{ij}(\nabla u_m)u_{m,ij},$
with  coefficients $a_{ij}, 1\leq i,j\leq 2,$ given by 
\begin{align*}
 a_{ij}(z)=\frac{\delta_{ij}}{\mu_m(|p|^2)}-\frac{2z_iz_j\mu_m'(|z|^2)}{\mu_m^2(|z|^2)} \qquad \text{ for $z=(z_1,z_2)\in \R^2.$}
\end{align*}
 Given $z\in\R^2,$ the eigenvalues of the matrix  $(a_{ij}(z))_{1\leq i,j\leq 2}$ are
\begin{align*}
 \lambda_1(z)=\frac{1}{\mu_m(|z|^2)}, \quad \lambda_2(z)=\frac{1}{\mu_m(|z|^2)}-\frac{2|z|^2\mu_m'(|z|^2)}{\mu_m^2(|z|^2)},
\end{align*}
and it follows readily from the fact that $\mu_m$ satisfies the relations \eqref{viscosity'} that the eigenvalues are positive and uniformly bounded in $z\in\R^2$,
i.e. $\Q$ is a uniformly elliptic quasilinear  operator.

From now on, we set $\Omega_m:=\Omega_m(0),$ $\Omega_w:=\Omega_w(0)$, and will identify  the common boundary   $\Gamma_0:=\s\times \{0\}$ of $\Omega_m$ and $\Omega_w$
with the unit circle $\s.$ 
Given $f\in \V,$ we define the mapping $\phi_f:\Omega \to \Omega$ by the relation 
\begin{align*} 
 \phi_f(x,y):=(x, y+(1-y^2)f(x))\qquad\text{for $ (x,y)\in \Omega.$} 
\end{align*}
Since functions $f\in \V$ satisfy $||f||_{C(\s)}<1/2$, one can easily verify that the function 
\[
\0\ni (x,y)\to \left(x, \frac{1-\sqrt{1-4yf(x)+4f^2(x)}}{2f(x)} \right)\in\0
\]
is the inverse of $\phi_f,$ that is $\phi_f$ is a diffeomorphism  having the same regularity   as $f$.
Moreover, we have that
$\phi_f(\0_m)=\0_m(f)$, $\phi_f(\0_w)=\0_w(f),$  and $\phi_f(\G_0)=\G(f).$
To keep the notation simple, we shall also denote by $\phi_f$  the diffeomorphisms obtained by restricting   $\phi_f$ to $\0_m$  and $\0_w.$

Using these diffeomorphisms, we re-express  now the equations of \eqref{newpb} on the fixed reference domains $\Omega_m$, $\Omega_w,$ and their boundaries.
To this end,   we introduce for each $f\in\V$ the operator $\A_w(f): h^{2+\alpha}(\Omega_w)\to h^{\alpha}(\Omega_w)$ by setting
\begin{align*}
 \A_w(f)v_w:=  \left(\Delta (v_w\circ\phi_{f}^{-1})\right)\circ\phi_f \qquad\text{for $v_w\in h^{2+\alpha}(\Omega_w)$.}
\end{align*}
A simple computation yields that 
\begin{align*}
  \A_w(f)v_w =&  \frac{\p^2v_w}{\p x^2}+2\frac{f'(y^2-1)}{1-2yf}\frac{\p^2 v_w}{\p x \p y}+\left(\frac{f'^2(y^2-1)^2}{(1-2yf)^2}+
\frac{1}{(1-2yf)^2}\right)\frac{\p^2 v_w}{\p y^2}\\
&+ \left( \frac{f''(y^2-1)}{1-2yf}+\frac{4f'^2y(y^2-1)}{(1-2yf)^2}+\frac{2ff'^2(y^2-1)^2+2f}{(1-2yf)^3} \right) \frac{\p v_w}{\p y},
\end{align*}
for $v_w \in h^{2+\alpha}(\Omega_w).$
Consequently, one can  see that $\A_w(f)$ is uniformly elliptic and that it depends analytically on  $f$, that is $\A_w\in C^\omega(\V,\kL(h^{2+\alpha}(\Omega_w), h^{\alpha}(\Omega_w))).$ 
Furthermore, corresponding to the operator $\Q,$ we let $\A_m:\V\times h^{2+\alpha}(\0_m)\to h^{\alpha}(\0_m)$ be the operator with 
\begin{align*}
 \A_m(f,v_m):=  \left(\Q (v_m\circ\phi_{f}^{-1})\right)\circ\phi_f \qquad \text{for $ v_m\in h^{2+\alpha}(\Omega_m)$}.
\end{align*}
\begin{lemma}\label{L:1} The operator $\A_m$ is smooth $\A_m\in C^\infty(\V\times h^{2+\alpha}(\0_m), h^{\alpha}(\0_m))$ and has a quasilinear structure 
$\A_m(f,v_m)= b_{ij}(y,f, \nabla_f v_m) v_{m,ij}+b(y,f, \nabla_f v_m) v_{m,2}$  whereby 
\[
\nabla_f v_m:=\left(v_{m,1}+\frac{f'(y^2-1)}{1-2yf} v_{m,2}, \frac{1}{1-2yf}  v_{m,2}\right), 
\]
and the coefficients of the operator are given by
\begin{align*}
 b_{11}(y,f, \nabla_f v_m) =& a_{11}(\nabla_fv_m),\\
 b_{12}(y,f, \nabla_f v_m)  =& \frac{f'(y^2-1)}{1-2yf}a_{11}(\nabla_f v_m)+\frac{1}{1-2yf}a_{12}(\nabla_f v_m),\\
 b_{22}(y,f, \nabla_f v_m) = &\frac{f'^2(y^2-1)^2}{(1-2yf)^2}a_{11}(\nabla_fv_m)
+2\frac{f'(y^2-1)}{(1-2yf)^2}a_{12}(\nabla_f v_m)+\frac{1}{1-2yf}a_{22}(\nabla_f v_m),\\
 b(y,f, \nabla_f v_m) = &\left[ \frac{f''(y^2-1)}{1-2yf}+\frac{4f'^2y(y^2-1)}{(1-2yf)^2}+\frac{2ff'^2(y^2-1)^2}{(1-2yf)^3}\right] a_{11}(\nabla_fv_m)\\
&+\frac{4f'y-4ff'(1+y^2)}{(1-2yf)^3}a_{12}(\nabla_fv_m)+\frac{2f}{(1-2yf)^3}a_{22}(\nabla_fv_m),
\end{align*}
for $ f\in \V, $ $v_m\in h^{2+\alpha}(\Omega_m).$
\end{lemma}

\begin{proof}
The quasilinear structure of $\A_m$ is found by using the definition of $\Q$ and $\A_m.$ 
Then, taking into account that $\mu_m$ is smooth, we deduce that $\A_m$ is a smooth function of both variables. 
\end{proof}

Corresponding to the fifth equation of system \eqref{newpb}, we define the boundary operators
 $\B_w(f):h^{2+\alpha}(\Omega_w)\to {\it h}^{1+\alpha}(\s)$  and  $\B_m:\V\times h^{2+\alpha}(\Omega_m)\to {\it h}^{1+\alpha}(\s)$
by setting
\begin{align*}
&B_w(f)v_w=\frac{1}{\mu_v} \tr \left(\langle \nabla (v_w\circ\phi_{f}^{-1})   | (-f',1) \rangle\circ \phi_f\right), \quad v_w\in h^{2+\alpha}(\Omega_w),\\
&B_m(f,v_m)=  \tr \left(\left\langle \frac {\nabla (v_m\circ\phi_{f}^{-1})   }{\mu_m (|\nabla (v_m\circ\phi_{f}^{-1})|^2)}  | (-f',1) \right\rangle\circ \phi_f\right), \quad (f,v_m)\in \V\times h^{2+\alpha}(\Omega_m),
\end{align*}
respectively, where $\tr$ is the trace operator with respect to $\s=\Gamma_0.$ 
More precisely, we have that 
\begin{align}\label{Bw}
 \B_w(f)v_w=\mu_w^{-1}\left[(1+f'^2)\tr v_{w,2}-f' \tr v_{w,1} \right]  \qquad\text{for $v_w\in h^{2+\alpha}(\Omega_w),$}
\end{align}
so that $\B_w(f)$ depends analytically on $f,$ i.e. $\B_w \in C^\omega(\V, \kL(h^{2+\alpha}(\Omega_w), h^{1+\alpha}(\s))).$
Concerning the  boundary operator $\B_m$, we get
\begin{align}\label{Bm}
 \B_m(f, v_m)= \displaystyle\frac{(1+f'^2)  v_{m,2}-f'\tr  v_{m, 1}}{\mu_m\left( 
\tr |\nabla _fv_m|^2
\right)} \qquad\text{for $(f,v_m)\in\V\times  h^{2+\alpha}(\Omega_m)$,}
\end{align}
whereby $\tr \nabla _fv=(v_{m,1}-f' v_{m,2}, v_{m,2}).$
From \eqref{viscosity'}  together with the smoothness   of $\mu_m,$ we deduce that $\B_m$ is a smooth operator 
$ \B_m \in C^\infty(\V\times h^{2+\alpha}(\Omega_m), h^{1+\alpha}(\s)).$
Letting $v_w=u_w\circ \phi_f$ and $v_m=u_m\circ \phi_f,$ we see that the tuple $(f,v_w,v_m) $ satisfies the following system of equations
\begin{equation}\label{TP}
\left\{ 
\begin{array}{rllllll}
  \A_w(f)v_w&=&0 \qquad &\text{in} \quad \Omega_w, \\
  \A_m(f,v_m)&=&0 \qquad &\text{in} \quad \Omega_m,\\
  v_w&=&h(t,x) \qquad &\text{on} \quad \Gamma_1,\\
  v_{m,2}&=&0 \qquad &\text{on} \quad \Gamma_{-1},\\
  v_w-v_m&=& \gamma \kappa(f)+g(\rho_w-\rho_m)f \qquad &\text{on} \quad \Gamma_0,\\
  \p_t f=-\B_w(f)v_w&=& -\B_m(f,v_m) \qquad &\text{on} \quad \Gamma_0,\\
  f(0)&=& f_0,
 \end{array}
\right.
\end{equation}
pointwise. 
When $\gamma>0,$ we set  $\kappa: \V\to h^{2+\alpha}_0(\s)$ to be defined by $\kappa(f)= f''/(1+f'^2)^{3/2}$ for $f\in \V$.
Let us observe that the curvature operator $\kappa$ is also a real-analytic function
 $\kappa\in C^\omega(\V, h^{2+\alpha}_0(\s))$.
 We enhance that the equations of \eqref{TP} are all defined on sets which are fixed in time.
 
Defining the notion of classical solutions for \eqref{TP} in a similarly way to  that for \eqref{newpb}, we have the following equivalence between 
 problems \eqref{newpb} and \eqref{TP}.
\begin{lemma}\label{equi}
If  $(f, u_w, u_m)$ is  a solution of \eqref{newpb}, then 
the tuple $(f, u_w\circ \phi_f, u_m\circ \phi_f)$ is a solution of \eqref{TP}.
Conversely, if $(f, v_w, v_m)$ is a solution of \eqref{TP}, then $(f, v_w\circ\phi_f^{-1},  v_m\circ\phi_f^{-1})$ is a  solution 
of \eqref{newpb}.
\end{lemma}
\begin{proof}
Assume that  $(f, u_w, u_m)$ is  a solution of \eqref{newpb}. 
As we already noticed, the tuple $(f, u_w\circ \phi_f, u_m\circ \phi_f)$  solves \eqref{TP} pointwise.
We show now that  $v_w:= u_w\circ \phi_f$ has the required regularity $v_w\in h^{2+\alpha}(\Omega_w).$
To this end, let $(f_n)\subset C^\infty(\s)$ be a sequence such that $f_n\to f$ in $C^{2+\alpha}(\s).$
We can choose $(f_n)$ such that in fact $f_n>f$ for all $n\in\N $  (we do not require that $f_n\in\V$).
If $(u_w^p)\in C^{\infty}(\ov{\Omega_w(f)})$ is a sequence such that $u_w^p \to u_w,$ then due to $\0_w(f_n)\subset \0_w(f)$,
the function $u_w^p\circ\phi_{f_n} $ is well-defined and $u_w^p\circ\phi_{f_n}\in C^{\infty}(\ov{\Omega_w}).$
The function $v_w$ is an accumulation point of the set $\{u_w^p\circ\phi_{f_n}\,:\, p,n\in\N\}:$
\begin{align*}
\|u_w^p\circ\phi_{f_n}-v_w\|_{C^{2+\alpha}(\ov{\Omega_w})}&\leq \|u_w^p\circ\phi_{f_n}-u_w^p\circ\phi_{f}\|_{C^{2+\alpha}(\ov{\Omega_w})}
+\|(u_w^p-u_w)\circ\phi_{f}\|_{C^{2+\alpha}(\ov{\Omega_w})}\\
&\leq C \left(\|u_w^p\|_{C^{4}(\ov{\Omega_w})}\|f_n-f\|_{C^{2+\alpha}(\s)}+
+\|u_w^p-u_w\|_{C^{2+\alpha}(\ov{\Omega_w})}\right),
\end{align*}
so that indeed $v_w\in h^{2+\alpha}(\Omega_w).$  
Analogously, $v_m\in h^{2+\alpha}(\Omega_m) $ and we proved the first claim. 
The  reciprocal implication follows by similar reasons. 
\end{proof}

\subsection{The abstract formulation of the problem}
We enhance that even if the transformation performed in the previous paragraph 
has the hindrance of introducing additional nonlinear coefficients, it allows us to re-write the problem
as a single nonlinear equation.
To do this, we define in the following solution operators to elliptic boundary value problems which are closely related to  system \eqref{TP}.
They are used  later on to  formulate the original problem \eqref{TP} as a single operator equation. 
\begin{lemma}\label{LS1}
Given $f\in \V$ and $(F, h)\in h^{1+\alpha}(\s)\times h^{2+\alpha}(\s),$ the problem
\begin{equation}\label{T}
 \left\{
\begin{array}{rllll}
 \A_w(f)v_w&=&0 \qquad &\text{in} \quad \Omega_w,\\
 \B_w(f)v_w&=& F \qquad &\text{on} \quad \Gamma_0,\\
 v_w&=& h \qquad &\text{on} \quad \Gamma_1.
\end{array}
\right.
\end{equation}
possess a unique solution  $\T(f)[F, h]\in h^{2+\alpha}(\Omega_w).$
Moreover, the operator $\T$ is real-analytic $\T\in C^\omega(\V, \kL(h^{1+\alpha}(\s)\times h^{2+\alpha}(\s), h^{2+\alpha}(\0_w))).$
\end{lemma}
\begin{proof} Given $f\in \V,$ the problem \eqref{T} possesses a unique classical solution  $u_w\circ\phi_f,$
where $u_w$ is the unique solution of the linear elliptic boundary value problem
\begin{equation}\label{elum}
\left\{
\begin{array}{rllllllllllll}
\Delta u_w&=& 0 \qquad &\text{in} \quad   \Omega_w(f), \\
\mu_w^{-1}\sqrt{1+f'^2}\p_\nu u_w&=& F\qquad &\text{on} \quad  \Gamma(f), \\
 u_{w} &=& h  \qquad &\text{on} \quad  \Gamma_{1}.
\end{array}
\right.
\end{equation}
Denoting by $\tr_1$ the trace operator with respect to $\Gamma_1,$ the operator  $\T(f)$ is the inverse of the linear operator
\[
(\A_w(f),\B_w(f), \tr_1):C^{2+\alpha}(\ov{\Omega_w})\to C^{\alpha}(\ov{\Omega_w})\times C^{1+\alpha}(\s)\times C^{2+\alpha}(\s).
\]
Since $\A_w$ and $\B_w$ depend analytically of $f$, and the function associating to a bijective operator its inverse is also real-analytic, 
we get $\T\in C^\omega(\V, \kL(C^{1+\alpha}(\s)\times C^{2+\alpha}(\s), C^{2+\alpha}(\ov{\0_w}))).$
Due to elliptic regularity, cf. Theorem  in \cite{GT}, the solution $v_w$ of \eqref{T} belongs to $C^{\infty}(\ov{\Omega_w}) $ if $f,F,h$ are smooth functions.
By continuity, we conclude that $\T(f)[F, h]\in h^{2+\alpha}(\Omega_w) $ for all $f\in \V$ and $(F, h)\in h^{1+\alpha}(\s)\times h^{2+\alpha}(\s),$
which completes the proof.
\end{proof}  

We enhance that second equation of \eqref{T} corresponds to one of the kinematic boundary condition in \eqref{newpb}.
Whence, if we know $f$ and its derivative $\p_t f$, the function $v_w$ is given by $ v_w=\T(f)[-\p_tf,h].$
We can use now  $v_w$ as a boundary data in the fifth equation of \eqref{TP}  and determine $v_m$ as a function of $f$ and $\p_tf.$
\begin{lemma}\label{LS2}
Given $f\in \V$ and $p\in  h^{2+\alpha}(\s),$ the problem
\begin{equation}\label{TT}
 \left\{
\begin{array}{rllll}
 \A_m(f,v_m)&=&0 \qquad &\text{in} \quad \Omega_m,\\
 v_m&=& p \qquad &\text{on} \quad \Gamma_0,\\
 v_{m,2}&=&0 \qquad &\text{on} \quad \Gamma_{-1}
\end{array}
\right.
\end{equation}
possess a unique solution  $\cR(f, p)\in h^{2+\alpha}(\Omega_m).$
Moreover, the operator $\cR$ is smooth $\cR\in C^\infty(\V \times h^{2+\alpha}(\s), h^{2+\alpha}(\0_w)).$
\end{lemma}
\begin{proof}
We pick $f\in\V $ and  consider the quasilinear Dirichlet problem 
\begin{equation}\label{DP}
\left\{
 \begin{array}{rllllllllllll}
\Q u_m&=& 0 \qquad &\text{in} \quad  \Omega_{-2-f,f}, \\
u_m&=&p \qquad &\text{on} \quad  \p\Omega_{-2-f,f}.
  \end{array}
\right.
\end{equation}
Theorem 8.3 in \cite{La} together with condition \eqref{viscosity'} guarantees the solvability of \eqref{DP} in the Banach space  $C^{2+\alpha}(\ov{ \Omega_{-2-f,f}})$.
Furthermore, the uniqueness of the classical solution of \eqref{DP} is a consequence of the weak elliptic maximum principle for quasilinear elliptic equations, 
cf. Theorem 10.1 in \cite{GT}.  
Denoting  the unique solution of \eqref{DP} by  $u_m$, because of the symmetry of the domain $ \Omega_{-2-f,f}$
and of the boundary data, the  weak elliptic maximum principle ensures ensures also   that $u(x,y)=u(x,-2-y)$ for all $(x,y)\in  \Omega_{-2-f,f}$.
Particularly, $u_{m,2}=0$ on $\Gamma_{-1},$ meaning that $v_m:=\left(u_m\big|\0_m(f)\right)\circ\phi_{f}^{-1}\in C^{2+\alpha}(\ov{\Omega_m})  $ is a solution of \eqref{DP}, which is unique by the same result, Theorem 10.1 in \cite{GT}. . 
Whence, for all $f\in\V$ and $p\in  h^{2+\alpha}(\s),$ the problem \eqref{LS2} possesses a unique solution $\cR(f, p)\in C^{2+\alpha}(\ov{\Omega_m}).$
Since $\cR(f, p)\in C^{\infty}(\ov{\Omega_m})$ when both functions $f$ and $p$ are both smooth, if we show that $\cR$ is a smooth operator, then, by a density argument, 
we obtain that $\cR$ maps  into $ h^{2+\alpha}(\0_m).$

In order to prove this regularity property, we notice that $(\A_m(f,v_m), \tr v_m-p,\p_2v_m)=0$ if and only if $v_m:=\cR(f, p),$
where we consider
\[
 (f,p,v_m)\mapsto (\A_m(f,v_m), \tr v_m  -p,v_{m,2})\in C^{\alpha}(\ov{\Omega_m})\times C^{2+\alpha}(\s)\times C^{1+\alpha}(\s).
\]
as a map defined on $\V\times C^{2+\alpha}(\s)\times C^{2+\alpha}(\ov{\Omega_m}).$
Since $\A_m$ is  smooth, if we prove that $\p_{v_m}(\A_m(f,\cdot), \tr -p,\p_2)\big|_{v_m=\cR(f, p)}$ is an isomorphism, then by the implicit 
function theorem we deduce that $\cR$ is a smooth map and we are done.
From Lemma \ref{L:1} we obtain that $\p_{v_m}(\A_m(f,\cdot), \tr -p,\p_2)\big|_{v_m}=(T, \tr ,\p_2): C^{2+\alpha}(\ov{\Omega_m})\to C^{\alpha}(\ov{\Omega_m})\times C^{2+\alpha}(\s)\times C^{1+\alpha}(\s),$ 
where 
\begin{align*}
Tw_m:=&b_{ij}(y,f, \nabla_f v_m) w_{m,ij}+b(y,f, \nabla_f v_m) w_{m,2}+v_{m,ij}\p_{3} b_{ij}(y,f, \nabla_f v_m)\nabla_f w_m\\
&+v_{m,2}\p_3 b(y,f, \nabla_f v_m)\nabla_f w_m. 
\end{align*}
Hereby, $\p_3$ denotes   differentiation of the coefficient functions with respect to the third  variable $\nabla_f v_m$
Since $T$ is a linear elliptic operator and the coefficient of $w_m$ is zero, we conclude that  $(T,\tr,\p_2) $ is an isomorphism, and the proof is completed.
\end{proof}

With this notation,  we remark that $(f, v_w, v_m)$ is a solution to \eqref{TP} if and only if $f(0)=f_0,$
$v_w=\T(f)[-\p_tf,h],$  $v_m=\cR(f, \tr v_w-\gamma\kappa(f)+g(\h_m-\h_w)f),$ and 
\begin{align}\label{condsatis}
 \p_t f+ \B_m(f,\cR(f, \tr \T(f)[-\p_tf,h]-\gamma\kappa(f)+g(\h_m-\h_w)f))=0 
\end{align}
for all $t\in[0,T]$, for some $T>0$.
We shall now justify our choice of working with functions having integral mean zero.
\begin{lemma} [The volumes of both fluids are preserved] 
Given $(f,F)\in h^{2+\alpha}_0(\s)\times h^{1+\alpha}_0(\s), $ we have that
\begin{equation*}
F+ \B_m(f,\cR(f, \tr \T(f)[-F,h]-\gamma\kappa(f)+g(\h_m-\h_w)f))\in h^{1+\alpha}_0(\s).
\end{equation*}
\end{lemma}
\begin{proof} Let $v_w:=\tr \T(f)[-F,h]$, $v_m:=\cR(f, \tr v_w-\gamma\kappa(f)+g(\h_m-\h_w)f),$ $u_w:=v_w\circ\phi_f^{-1}$
and  $u_m:=v_m\circ\phi_f^{-1}.$
Having assumed \eqref{bot}, we find from  Stokes'  formula that
\begin{align*}
\int_\s \B_m(f,v_m)\, dx&=\int_\s \left\langle \frac {\nabla u_m   }{\mu_m (|\nabla u_m|^2)}  | (-f',1) \right\rangle(x,f(x))\, dx=\int_{\Gamma_f} \p_\nu\left(\frac {\nabla u_m   }{\mu_m (|\nabla u_m|^2)} \right)\, d\sigma\\
&=\int_{\0_m(f)} \nabla\left(\frac {\nabla u_m   }{\mu_m (|\nabla u_m|^2)} \right)\, dx=0,
\end{align*}
with $d\sigma$ denoting the curve integral.
This and our assumption $F\in  h^{1+\alpha}_0(\s) $ leads us to the desired conclusion.
\end{proof}

\section{Proof of the main result}\label{S4}
Resuming the last section, we have  reduced the original problem \eqref{newpb} to the following problem: given  $f_0\in\V$ 
find a positive time  $T=T(f_0)>0$ and solution curve $f$ with $f\in C([0,T),\V)\cap C^1([0,T), \Y)$ 
and 
\begin{equation}\label{eq:PB}
\cF(h(t),f(t),\p_tf(t))=0 \qquad\text{for all $t\in [0,T).$}
\end{equation}
Hereby, $\cF: h^{2+\alpha}(\s)\times \V\times \Y\to \Y$
is the nonlocal operator given by
\begin{equation}\label{eq:PBF}
\cF(h,f,F):=F+ \B_m(f,\cR(f, \tr \T(f)[-F,h]-\gamma\kappa(f)+g(\h_m-\h_w)f)).
\end{equation}
We note that the operator \eqref{eq:PBF} is fully nonlinear in each of its variables.

First, we observe that the tuple $(\ov h,0,0)$ is a solution of the equation $\cF(h,f,F)=0$ for any constant  $\ov h\in\R.$ 
Using the implicit function theorem we may express $F$ as a function of $(h,f)$ at least in a neighborhood of the point
$(\ov h  ,0,0),$ cf. Lemma \ref{Lema1}. 
Doing this, we  arrive at a fully nonlinear abstract evolution equation. 
Indeed, if $\ov h\in\R,$ and $(f,F)=(0,0)$ then problem  \eqref{T} has the constant solution  $\T(f)[-F,\ov h]=\ov h$, and 
therefore we also have $\cR(f, \tr \T(f)[-F,h]-\gamma\kappa(f)+g(\h_m-\h_w)f)=\ov h$.
Whence,  $(\ov h,0,0)$ is  indeed a   solution of   $\cF(h,f,F)=0$.
\begin{lemma}\label{Lema1} 
Given $\ov h\in\R$, the Fre\'chet derivative  $\p_F\cF(\ov h,0,0)$ is an isomorphism, that is
\[
 \p_F\cF(\ov h,0,0)\in{\rm Isom}(\Y).
 \]
 \end{lemma}
\begin{proof} Using the chain rule and the linearity of the solution operator $\T$ with respect to $(F,h)$,
we find that
\begin{align}\label{1}
\p_F\cF(\ov h,0,0)&={\rm id}_{\Y}-\p_{v_m}\B_m(0,\ov h)[\p_p\cR(0, \ov h)[\tr \T(0)[\cdot,0]]].
\end{align} 
It readily follows from \eqref{Bm} that  we have
\begin{align}\label{2}
 \p_{v_m}\B_m(0,\ov h)[w_m]= \frac{\tr w_{m,2}}{\mu_m(0)}\qquad\text{for all $w_m\in h^{2+\alpha}(\0_m).$}
\end{align}
Further on, using the definition of the operator $\cR$ and   Lemma \ref{L:1} we get that $w_m:= \p_p\cR(0, {\ov h})[q]$ is the solution of the problem
\begin{equation}\label{s1}
\left\{ 
\begin{array}{rlllll}
  \Delta w_m&=&0 \qquad &\text{in} \quad \Omega_m\\
  w_m&=&q \qquad &\text{on} \quad \Gamma_0\\
  w_{m,2} &=& 0 \qquad &\text{on} \quad \Gamma_{-1}
 \end{array}
\right.
\end{equation}
for all $q\in h^{2+\alpha}(\s).$
Pick now $q\in h^{2+\alpha}(\s) $ and consider its Fourier expansion $q=\sum_{k\in \Z} b_ke^{ikx}.$ 
Expanding also $w_m$ as  $w_m=\p_p\cR(0, {\ov h})[q]= \sum_{k\in \Z} B_k(y)e^{ikx}$   we find by inserting both expressions into \eqref{s1}
 and comparing the coefficients of $e^{ikx}$ for every $k,$ that $B_k$ is the solution of
\begin{equation*}
\left\{ 
\begin{array}{rlllll}
  B_k(y)''-k^2B_k(y)&=&0  \qquad &-1<y<0\\
  B_k(0)&=&b_k \\
  B_k'(-1)&=&0
 \end{array}
\right.
\end{equation*}
and therefore
\begin{align}\label{3}
 B_k(y)=\left(\frac{ e^k}{e^k+e^{-k}}e^{ky}+\frac{e^{-k}}{e^k+e^{-k}}e^{-ky}\right) b_k, \qquad k\in\Z.
\end{align}
Next, we consider the Fourier expansion of $F= \sum_{k\in \Z\setminus\{0\}} a_k e^{ikx}\in\Y.$ 
Similarly as above, we obtain from the definition of the operator $\T$ that 
\begin{align}\label{4}
 \tr \T(0)[F,0]=\sum_{k \in \Z\setminus\{0\}} -\mu_w\frac{ \tanh(k)}{k} a_ke^{ikx}.
\end{align}
 Gathering the relations \eqref{3} and \eqref{4}, we obtain from \eqref{1} and \eqref{2} that
\begin{align}\label{pF} 
 \p_F\cF(\ov h,0,0)[F]=\sum_{m\in \Z\setminus\{0\}} \left( 1+\frac{\mu_w}{\mu_m(0)}\tanh^2(k) \right)a_k e^{ikx} \quad \text{for}\quad F= \sum_{k\in \Z} a_k e^{ikx}.
\end{align}
Whence, we have shown that $\p_F\cF(\ov h,0,0)$ is a Fourier multiplier with symbol $m(k):=1+(\mu_w/\mu_m(0))\tanh^2(k), k\in\Z\setminus\{0\}$. 
Since, $\sup_{k\in\Z}|1/m(k)|<\infty,$ we infer from \cite[Theorem 3.6.3]{ST} that the Fourier multiplier with symbol $1/m$  belongs to  $\kL(C^{p+\alpha}(\s))$ for all $p\in\N.$  
From this, it follows at once that  
it  belongs also to $\kL(\Y)$ and that it is the inverse of  $\p_F\cF(\ov h,0,0)$.
\end{proof}

Using Lemma \ref{Lema1}, we obtain the following reformulation of the   equation  $\cF(h,f,F)=0$.

\begin{prop}\label{P:4}
Given $\ov h\in\R,$ there exist an open neighborhood $\ww\times\oo\times \zz\subset h^{2+\alpha}(\s)\times \V\times \Y$   of $(\ov h,0,0)$  and a smooth operator 
$\Phi:\ww\times\oo\subset h^{2+\alpha}(\s)\times \X \to \Y$ with the property that 
if $(h,f,F)\in \ww\times\oo\times \zz$ is a solution of the equation $\cF(h,f,F)=0$  then
\begin{equation}\label{CP}
 F=\Phi (h, f).
\end{equation}
Moreover, if $\mu_m'\geq0,$ then we can choose $\zz=\Y.$
\end{prop}

\begin{proof}
Since
$\cF(\ov h, 0,0)=0$, 
$\cF \in C^{\infty}( h^{2+\alpha}(\s)\times \V\times \Y, \Y ),$ and 
$\p_{F}\cF (\ov h, 0, 0)\in {\rm Isom}(\Y)$, cf. Lemma \ref{Lema1},
 we obtain from the implicit function theorem that there exist an open neighborhood $\ww\times \oo    \times \zz \subset h^{2+\alpha}(\s)\times \V \times \Y$ 
 of $(\ov h,0,0)$  and $\Phi \in C^\infty(\oo \times \ww, \zz)$ such that 
\[
\cF(f,h, F)=0 \quad \text{and} \quad (f,h,F)\in \oo \times \ww \times \zz \quad \text{implies} \quad  F=\Phi(f,h).
\]
In the rest of the proof, we show that for shear thickening fluids, we have $\zz=\Y.$ 
To this end, let  $f\in \V,$ $h\in h^{2+\alpha}(\s),$ and   $F_1, F_2 \in \Y$ be given such that 
\[
 \cF(h,f,F_1)=\cF(f,h,F_2)=0.
\]
If
\begin{align*}
 &u_m^i=\T(f)[-F_i,h)]\circ \phi_f^{-1}, \qquad i=1,2 \\
 &u_w^i=\cR(f, \tr \T(f)[-F_i,h]-\gamma \kappa(f)+g(\rho_m-\rho_w)f)\circ \phi_f^{-1}, \qquad i=1,2 
\end{align*} 
then, by the definition of $\T$,$\cR$, $\B_w$, and $\B_m,$ we find that $(u_m^i, u_w^i)$ solve the following system  of equations
\begin{equation}\label{newpbi}
\left\{
 \begin{array}{rllllllllllll}
\Delta u_w^i&=& 0 \qquad &\text{in} \quad   \Omega_w(f), \\
\nabla \left(\displaystyle\frac{\nabla u_m^i}{\mu_m(|\nabla u_m^i|^2)}\right)&=& 0 \qquad &\text{in} \quad  \Omega_m(f), \\
u_w^i &=& h \qquad &\text{on} \quad  \Gamma_{1}, \\
 u_{m,2}^i &=& 0  \qquad &\text{on} \quad  \Gamma_{-1}, \\
u_w^i-u_m^i&=& \gamma \kappa_{\Gamma(f)}+g(\rho_w-\rho_m)f \qquad &\text{on} \quad  \Gamma(f),\\
F_i= -\displaystyle\frac{1}{\mu_w}\sqrt{1+f'^2}\p_\nu u_w^i&=& -\displaystyle\frac{\sqrt{1+f'^2}}{\mu_m(|\nabla u_m^i|^2)}\p_\nu u_m^i
\qquad &\text{on} \quad  \Gamma(f).
  \end{array}
\right.
\end{equation}
Defining $U_m:=u_m^1-u_m^2$ and $U_w:=u_w^1-u_w^2$ we obtain by subtracting the corresponding equations of \eqref{newpbi}  for $i=1$ and $i=2$ that 
\begin{equation}\label{UU}
\left\{
 \begin{array}{rllllllllllll}
\Delta U_w&=& 0 \qquad &\text{in} \quad   \Omega_w(f), \\
\kL U_m&=& 0 \qquad &\text{in} \quad  \Omega_m(f), \\
U_w &=& 0 \qquad &\text{on} \quad  \Gamma_{1}, \\
 U_{m,2}&=& 0  \qquad &\text{on} \quad  \Gamma_{-1}, \\
U_w-U_m&=& 0  \qquad &\text{on} \quad  \Gamma(f),
  \end{array}
\right.
\end{equation}
where 
\begin{align*}
\kL U_m:= \Q u_m^1-\Q u_m^2&=a_{ij}(\nabla u_m^1)u_{m,ij}^1-a_{ij}(\nabla u_m^2)u_{m,ij}^2\\
&= a_{ij}(\nabla u_m^1)U_{m,ij}+u_{m,ij}^2 \langle c_{ij} | \nabla U_m\rangle
\end{align*}
and $c_{ij}:=\int_0^1 \p a_{ij}(t \nabla u_m^1+(1-t)\nabla u_m^2)\, dt, 1\leq i, j\leq 2.$
Let us assume  that $U_w$ is not vanishing and its maximum is positive (the case when $U_w$ possesses a 
negative minimum can be treated similarly).
Then, we  obtain from the strong maximum principle that the positive   maximum  of $U_w$ is attained on the boundary $\Gamma(f),$ say at $z_0=(x_0,f(x_0)).$
Because $U_m=U_w$ on  $\Gamma(f)$ the maximum of $U_m$ is also attained at $z_0.$ 
Hopf's maximum principle ensures then that 
\begin{align}\label{Ho}
 \p_\nu U_w(z_0)<0 \quad \text{and} \quad  \p_\nu U_m(z_0)>0.
\end{align}
Moreover, from the last equations in \eqref{newpbi} we see  that
\begin{align}\label{relmu}
 \frac{\p_\nu u_m^1}{\mu(|\nabla u_m^1|^2)}=\frac{1}{\mu_w}\p_\nu u_w^1 \quad \text{and} \quad \frac{\p_\nu u_m^2}{\mu(|\nabla u_m^2|^2)}=\frac{1}{\mu_w}\p_\nu u_w^2.
\end{align}
Letting $x:=\p_\nu u_m^1(z_0)$ and $y:=\p_\nu u_m^2(z_0)$, we get from the second inequality in  \eqref{Ho} that $x>y.$
Furthermore, since  $U_m$ attains its maximum at $z_0$ its tangential derivative $\p_\tau U_m$  vanishes at $z_0$, meaning that $\p_\tau u_m^1(z_0)=\p_\tau u_m^2(z_0)=:r\in\R.$
Whence,  $|\nabla u_m^1|^2(z_0)=x^2+r^2$, $|\nabla u_m^2|^2(z_0)=y^2+r^2$ and relations \eqref{relmu} and \eqref{Ho} yield
\begin{align}\label{xmu}
\frac{x}{\mu_m(x^2+r^2)} < \frac{y}{\mu_m(y^2+r^2)}.
\end{align}
But, this is contradicting the fact that the function $\R\ni x\mapsto x/\mu(x^2+r^2)$ is strictly increasing, as we have 
\begin{align*}
 \frac{d}{dx}\left(\frac{x}{\mu_m(x^2+r^2)}\right)&= \frac{\mu_m(x^2+r^2)-2x^2\mu_m'(x^2+r^ 2)}{\mu_m^2(x^2+r^2)}\\
&\geq \frac{\mu_m(x^2+r^2)-2(x^2+r^2)\mu_m'(x^2+r^2)}{\mu_m^2(x^2+r^2)}\overset {\eqref{viscosity'}}{\geq} 0,
\end{align*}
where the first inequality holds true only if $\mu_m'\geq 0.$ 
Thus, $U_w\equiv 0$ and the last equation of \eqref{newpbi} leads us to the desired conclusion $F_1  = F_2$.
\end{proof}

We finish this section with the proof of our main result, Theorem \ref{thmain}.
The proof follows  by making use of  the abstract theory for fully nonlinear parabolic as presented in Chapter 8 of \cite{L}.
\begin{proof}[Proof of Theorem \ref{thmain}] Let $\alpha\in(0,1)$, $\ov h\in\R,$ and  $h$ be given such that \eqref{bcthm} is satisfied.
Then, we infer from Proposition \ref{P:4} that solving the original problem \eqref{newpb} is equivalent, for small initial and boundary data, to solving the fully nonlinear problem
\begin{equation}
 \p_tf=\Phi(t,f),\qquad f(0)=f_0,
\end{equation}
  whereby $\Phi(t,f):=\Phi(h(t),f)$ for all $(t,f)\in[0,T_0]\times \oo$, with $T_0>0$  chosen such that $h(t)\in \ww$ for all $t\in[0,T_0].$ 
Since $h$ is continuous, we obtain from Proposition \ref{P:4} that $\Phi$ and $\Phi_f$ are continuous mappings, that is $\Phi, \Phi_f\in C([0,T_0]\times \oo,\Y).$
Taking into account that the above considerations are true for any $\beta\in(0,\alpha)$ and that the small H\"older spaces satisfy the following property
\[
 (h^{\sigma_0}(\s), h^{\sigma_1}(\s))_\theta=h^{(1-\theta)\sigma_0+\theta\sigma_1}(\s)
\]
if $\theta \in (0,1)$ and $(1-\theta)\sigma_0+\theta \sigma_1\in \R^+ \setminus \N,$ where $(\cdot,\cdot)_\theta$  
denotes the continuous interpolation method of DaPrato and Grisvard \cite{DaP} (see also \cite{L}),
the proof of the theorem reduces to showing that $\p_f\Phi(\ov h,0)$ is the  generator of a strongly continuous analytic semigroup in $\kL(\Y).$ 

In order to prove this property, we infer from Proposition \ref{P:4} that 
\begin{align}\label{formdtp}
 \p_f \Phi(\ov h,0)=-[\p_F \cF(\ov h, 0,0)]^{-1}\circ \p_f\cF(\ov h,0, 0),
\end{align}
where $\p_F \cF(\ov h, 0,0)$ is the isomorphism given by \eqref{pF}.
We determine now a representation for $\p_f\cF(\ov h,0, 0)$.
Differentiating \eqref{eq:PBF} with respect to $f,$ yields 
\begin{equation}\label{frefP}
\p_f \cF(\ov h, 0,0)[f]= \p_{v_m}\B_m(0, \ov h)[\p_p \cR(0, \ov h)[-\gamma \p \kappa(0)[f]+g(\rho_m-\rho_w)f]]\qquad \text{for all $f\in\X.$}
\end{equation}
Indeed, this simplified expression is a consequence of the fact that $ \T(f)[0, \ov h]=\ov h$, $\cR(f,\ov h)=\ov h$,  and $\B_m(f,\ov h)=0$ for all $f\in  \V.$ 
Taking into account that $\p \kappa(0)[f]=f''$ for all $f\in\X$ and expending $f=\sum_{k \in \Z\setminus\{0\}}  a_k e^{ikx}$,
the relations \eqref{3} and \eqref{pF} lead us to the following representation
  \begin{align}\label{dfPhi}
 \p_f \Phi(\ov h, 0)f= \sum_{k \in \Z\setminus\{0\}} \lambda_k a_k e^{ikx}\qquad\text{for $f=\sum_{k \in \Z\setminus\{0\}}  a_k e^{ikx}$}
\end{align}
where the symbol $(\lambda_k)$ is given by
\begin{align}\label{lk}
 \lambda_k= \frac{k \tanh(k)}{\mu_m(0)+\mu_m \tanh^2(k)}[g(\rho_w-\rho_m)-k^2 \gamma], \qquad k\in \Z\setminus\{0\}.
\end{align}
The desired generator property may be obtained now from Theorem 3.6.3 in \cite{ST} (or Theorem 3.4 in \cite{EM1}).
We have verified therewith all the assumptions of Theorem 8.4.1 in \cite{L} and the desired statement Theorem \ref{thmain} follows from this abstract result.
\end{proof}

\vspace{0.5cm}
\hspace{-0.5cm}{\large \bf Acknowledgement}\\[2ex]
This research has been supported by the German Research Foundation (DFG) under the
grant ES 195/5-1.


\begin{thebibliography}{20}
\bibitem{A} {\sc  H. Abels}: {\em Existence of weak solutions for a diffuse interface model for viscous, incompressible
fluids with general densities}, Commun. Math. Phys. {\bf 289} (2009), 45-73.


 \bibitem{BP} {\sc D. Bothe \& J. Pr\"uss}: {\em $L_p$-Theory for a class of non-Newtonian fluids},  SIAM J. Math. Anal., 
 {\bf 39} (2)(2007), 379--421.

 \bibitem{CCG}
{\sc A. C\'ordoba, D. C\'ordoba \& F. Gancedo}: {\em Interface evolution: the Hele-Shaw and Muskat problems}, 
  Ann. Math. 173 (2011), 477--542.

\bibitem{DL} 
{\sc L. A. Dalrymple \& P. L.-F. Liu}: {\em Waves over soft mud: A two layer fluid model}, J. Phys. Oceanogr.,
 {\bf 8}  (1978), 1121--1131.  


\bibitem{DaP} {\sc  G. DaPrato \& P. Grisvard}: {\em Equations d'\'evolution abstraites nonlin\'eaires de type
parabolique},  Ann. Mat. Pura Appl. ${\bf{4}}$ (120)  (1979), 329--396.
 
\bibitem{EM1} 
{\sc  J. Escher \&  B.--V. Matioc}: {\em A moving boundary problem for periodic Stokesian Hele-Shaw flows},
 Interfaces Free Bound., {\bf{11}} (2009), 119--137.

\bibitem{EM}
{\sc J. Escher  \& B.-V. Matioc}: {\em On the parabolicity of the Muskat problem: well-posedness, fingering, and
stability results}, Z. Anal. Anwend.  {\bf 30} (2011), 193--218.
 
 \bibitem{GT}  
{\sc D. Gilbarg \& T. S. Trudinger}: {\em Elliptic Partial Differential Equations of Second Order}, Springer-Verlag, New York, 2001.



\bibitem{Hsu}
{\sc T.-J. Hsu, C. E. Ozdemir \& P. A. Traykovski}:
{\em  High-resolution numerical modeling of wave-supported gravity-driven mudflows}, J. Geophys. Res. {\bf 114}  (2009), C05014.
 
\bibitem{KSP} 
{\sc L. Konic, M. J. Shelley \& P. Palffy-Muhoray}: 
{\em Models of non-Newtonian Hele-Show flow}, 
Phys. Rev., E {\bf{54}}(5) (1996), R4536--R4539.

\bibitem{KW} 
{\sc C. Kranenburg \& J.~C.~Winterwerp}:
{\em  Erosion of fluid mud layers. I: Entrainment model}, Journal of Hydraulic Engineering  {\bf 123} (6) (1997), 504--511.

\bibitem{La}  
{\sc O. A. Ladyzhenskaya \& N. N. Ural'tseva Trudinger}: {\em Linear and quasilinear elliptic equations}, Academic Press, New York, 1968.

\bibitem{MS}
{\sc S. M\"uenzenmeier \& G. Starke}:
{\em First order system least square for coupled Stokes-Darcy Flow}
SIAM J. Numer. Anal. {\bf 49 } (2011), 387--404.


\bibitem{PS} 
{\sc J. Pr\"uss \& G. Simonett}:
{\em On the two-phase Navier-Stokes equations with surface tension},
Interfaces Free Bound.  {\bf 10} (2010), 311--345.

\bibitem{PS1} 
{\sc J. Pr\"uss \& G. Simonett}:
{\em On the Rayleigh-Taylor instability for the two-phase Navier-Stokes equations with surface tension},
Indiana Univ. Math. J., to appear.

\bibitem{L}  
{\sc Lunardi, A.}: {\em Analytic Semigroups and Optimal Regularity in Parabolic Problems},  Birkh\"auser, Basel, 1995.


\bibitem{ST}  
{\sc Schmeisser, H.-J. \& Triebel, H.}:  {\em Topics in Fourier Analysis and Function Spaces},  Wiley, New York, 1987.


\bibitem{SF}
{\sc E.~J.~Strang \& H.~J.~S.~Fernando}: {\em Entrainment and mixing in stratified shear flows},  J. Fluid Mech. {\bf 428} (2001), 349--386.


\bibitem{T}
{\sc E.~A.~Toorman}: {\em Modelling the thixotropic behaviour of dense cohesive sediment suspensions},  Rheol Acta {\bf 36} (1997), 56--65.

\bibitem{W}
{\sc J.C. Winterwerp}:
{\em On the flocculation and settling velocity of estuarine mud},
Continental Shelf Research {\bf 22} (2002), 1339--1360.




\bibitem{WG}
{\sc J. C. Winterwerp, R.F. de Graaff, J. Groeneweg, \& A.P. Luijendijk}: {\em Modelling of wave damping at
Guyana mud coast}, Coastal Engineering {\bf 54}(3) (2006), 249--261.





\bibitem{Y} 
{\sc   F. Yi}: {\em Local classical solution of Muskat free boundary problem},
   J. Partial Diff. Eqs., {\bf{9}} (1996), 84--96. 








\end{thebibliography}
\end{document}